\documentclass[11pt,reqno]{article}

\usepackage[margin=1in]{geometry}
\usepackage{amsmath,amsthm,amssymb}
\usepackage{xfrac}

\usepackage{setspace}

\usepackage[colorlinks=true, pdfstartview=FitV, linkcolor=blue,citecolor=blue, urlcolor=blue]{hyperref}
\allowdisplaybreaks

\usepackage[abbrev,lite,nobysame]{amsrefs}
\usepackage{times}
\usepackage[usenames,dvipsnames]{color}

\usepackage{mathtools,enumitem}

\usepackage[compact]{titlesec}

\usepackage[title]{appendix}

\usepackage{comment}

\definecolor{labelkey}{rgb}{0,0,1}

\def\tildeustart{\tilde u_{0}}
\def\tildeBstart{\tilde B_{0}}
\def\tildeuend{\tilde u_{1}}
\def\tildeBend{\tilde B_{1}}
\def\Btrunc{B_{\text T}}

\definecolor{mygray}{rgb}{.6, .6, .6}

\def\llabel#1{\notag}
\def\comma{\,,\qquad{}}

\def\indeq{\quad{}}

\newcommand{\R}{\mathbb{R}}

\newcommand{\tensor}{\otimes}

\newcommand{\cC}{\mathcal{C}}

\DeclareMathOperator{\supp}{\mathrm{supp}}

\DeclareMathOperator{\Div}{\mathrm{div}}

\DeclareMathOperator{\curl}{\mathrm{curl}}

\usepackage{bbm}

\newtheorem{theorem}{Theorem}[section]
\newtheorem{proposition}[theorem]{Proposition}

\newtheorem{lemma}[theorem]{Lemma}
\newtheorem*{lemma*}{Lemma}

\theoremstyle{definition}

\newcommand{\ustart}{u_{0}}
\newcommand{\Bstart}{B_{0}}
\newcommand{\uend}{u_{1}}
\newcommand{\Bend}{B_{1}}
\newcommand{\onen}{_{1,n}}
\newcommand{\twon}{_{2,n}}
\newcommand{\onenminus}{_{1,n-1}}
\newcommand{\twonminus}{_{2,n-1}}
\newcommand{\onenplus}{_{1,n+1}}
\newcommand{\twonplus}{_{2,n+1}}
\newcommand{\ushear}{u_{\textnormal{shear}}}
\newcommand{\Bshear}{B_{\textnormal{shear}}}
\newcommand{\pshear}{p_{\textnormal{shear}}}

\numberwithin{equation}{section}

\usepackage{etoolbox}
\usepackage{hyperref}

\makeatletter
\newcounter{author}
\renewcommand*\author[1]{%
  \stepcounter{author}%
  \ifnum\c@author=1
    \gdef\@author{#1}%
  \else
    \xdef\@author{\unexpanded\expandafter{\@author\and#1}}%
  \fi
  \csgdef{author@\the\c@author}{#1}}
\newcommand*\email[1]{%
  \csgdef{email@\the\c@author}{#1}}
\newcommand*\address[1]{%
  \csgdef{address@\the\c@author}{#1}}
\AtEndDocument{%
  \xdef\author@count{\the\c@author}%
  \c@author=1
  \print@authors}
\newcommand*\print@authors{%
  \ifnum\c@author>\author@count
  \else
    \print@author{\the\c@author}%
    \advance\c@author by 1
    \expandafter\print@authors
  \fi}
\newcommand*\print@author[1]{%
  \par\medskip
  \begin{tabular}{@{}l@{}}%
    \textsc{\csuse{address@#1}}\\
    \textit{E-mail address}:
    \href{mailto:\csuse{email@#1}}{\csuse{email@#1}}
  \end{tabular}}
\makeatother

\author{Igor Kukavica}
\address{\small Department of Mathematics, University of Southern California, Los Angeles, CA 90089}
\email{kukavica@usc.edu}

\author{Matthew Novack}
\address{\small Courant Institute of Mathematical Sciences, New York University, New York, NY 10012}
\email{mdn7@cims.nyu.edu}

\author{Vlad Vicol}
\address{\small Courant Institute of Mathematical Sciences, New York University, New York, NY 10012}
\email{vicol@cims.nyu.edu}

\title{{\bf Exact Boundary Controllability for the Ideal Magneto-hydrodynamic Equations}}
\date{}

\begin{document}

\maketitle
\begin{abstract}
We address the problem of controllability of the MHD system in a rectangular domain
with a control prescribed on the side boundary. We identify a necessary and sufficient
condition on the data to be null-controllable, i.e., can be driven to the zero state.
We also show that the validity of this condition allows the states to be stirred to each other.
If the condition is not satisfied, one can move from one state to another with the help of a
simple shear external magnetic force.
\end{abstract}

\section{Introduction}
We consider the two- or three-dimensional ideal magneto-hydrodynamic (MHD) equations for the unknown velocity $u\colon\Omega\times[0,T]\rightarrow\R^d$, magnetic field $B\colon\Omega\times[0,T]\rightarrow\R^d$, and pressure $p\colon\Omega\times[0,T]\rightarrow\R$:
\begin{subequations}
\label{eq:mhd}
\begin{align}
\partial_t u + u\cdot\nabla u - B\cdot\nabla B + \nabla p &= 0 \\
\partial_t B + u \cdot \nabla B - B\cdot\nabla u &= 0 \label{eq:B:equation} \\
\Div u = \Div B &= 0  \, .
\end{align}
\end{subequations}
In general, $\Omega\subset \R^d$ for $d=2$ or $3$ is an arbitrary set with Lipschitz boundary $\partial \Omega$ and outward pointing unit normal vector $n=n(x)$ for $x\in\partial\Omega$, in which case \eqref{eq:mhd} is classically supplemented with the boundary conditions
\begin{equation}\notag
    u \cdot n = B \cdot n = 0 
    \quad\textnormal{on}\quad \partial\Omega
    \, .
\end{equation}
For the existence theory of classical solutions to the ideal MHD system \eqref{eq:mhd}, see e.g.~\cite{Caflisch97}.

\subsection{Main result and comments}

To set the notation, let $\Gamma$ be an open and non-empty subset of $\partial\Omega$ which has non-empty intersection with every connected component of $\partial\Omega$. The question of exact boundary controllability of \eqref{eq:mhd} may be stated as follows. Given $T>0$, initial data $(\ustart,\Bstart)$, and terminal data $(\uend,\Bend)$ such that
\begin{subequations}\label{eq:data:conditions}
\begin{align}
    \Div \ustart = \Div \uend = \Div \Bstart = \Div \Bend &= 0 \quad\textnormal{in}\quad \Omega\\ 
    \ustart \cdot n = \uend\cdot n = \Bstart \cdot n = \Bend \cdot n &= 0 \quad\textnormal{on}\quad \partial\Omega\setminus\Gamma \, ,
\end{align}
\end{subequations}
does there exist a solution $(u,B)$ of the MHD equations \eqref{eq:mhd} such that 
\begin{subequations}\label{eq:solution:conditions}
\begin{align}
    (u,B)|_{t=0} &= (\ustart,\Bstart) \\ 
    (u,B)|_{t=T} &= (\uend,\Bend)\\
    u\cdot n = B &\cdot n = 0, \qquad  t\in[0,T],\quad  x\in\partial\Omega\setminus\Gamma \, ?
\end{align}
\end{subequations}
{In full generality, the answer to this question is no, as we demonstrate here. This contrasts sharply with the case of the incompressible Euler equations, in which the boundary control problem was first addressed in the two-dimensional case by Coron \cite{coron2,coron1} and then by Glass in the fully general three-dimensional setting~\cite{Glass00}. In this paper, we prove the exact boundary controllability for the MHD equations posed in a simple type of domain {given that certain extra conditions are satisfied}.  After the statement of the main theorem, we provide some simple calculations indicating that in many scenarios, these conditions are necessary and in fact sharp.}

Throughout the paper, we set $d=2$ and $\Omega=[0,1]^2$. The controlled portion $\Gamma$ of the boundary is the set $\{x=0,1\}\times(0,1)$, and we impose impermeability boundary conditions for $u=(u^1,u^2)$ and $B=(B^1,B^2)$ on $(0,1)\times\{y=0,1\}$. This is the setting of MHD in a planar duct, as considered in a recent preprint by Rissel and Wang \cite{risselwang}. We prove the following theorem.

\begin{theorem}\label{th:main}
Suppose $(\ustart,\Bstart)$ and $(\uend,\Bend)$ are $H^r$ regular
divergence-free vector fields,
where $r\geq 3$ is an integer,
with vanishing normal components on 
$(0,1)\times\{0,1\}$,
and assume that $B_0$ and $B_1$ satisfy
\begin{equation}\label{eq:mean}
    \int_{[0,1]^2} \Bstart^1 \,dx\,dy = \int_{[0,1]^2} \Bend^1 \,dx\,dy = 0 \, .
\end{equation}
Then there exists a 
solution $(u,B)\in C([0,T]; H^r([0,1]^2))$ 
to \eqref{eq:mhd} 
such that \eqref{eq:data:conditions} and \eqref{eq:solution:conditions} hold. If \eqref{eq:mean} is \emph{not} satisfied, then the same theorem holds but with a forcing term ${h}(t)e_x$ in the equation \eqref{eq:B:equation} for the magnetic field $B$; that is, \eqref{eq:B:equation} must be modified as
\begin{equation}\notag
\partial_t B + u \cdot \nabla B - B \cdot \nabla u = h(t)e_x \, .
\end{equation}
\end{theorem}

From here on, we fix $r\in\{3,4,5,\ldots\}$ and note that non-integer
values $r>2$ can be covered using the same method.

The implementation of the condition \eqref{eq:mean} is new  and provides the first instance of a solution to the MHD control problem \emph{without} a bulk forcing term in the equation for the magnetic field, and it characterizes some scenarios where a bulk forcing term is necessary.
We note that Rissel and Wang \cite{risselwang} used a forcing term in the equation for $B$ which is a harmonic function but which is not as simple or as explicit as the forcing term $h(t)e_x$. 
Also, in \cite{risselwang} the forcing term is present regardless
of the validity of the condition \eqref{eq:mean}.

The necessity of \eqref{eq:mean} may be seen from the following reasoning.
Suppose that the pair $(u,B)$ solves \eqref{eq:mhd} on the set $\Omega=[0,1]^2$, and $u^2|_{y=0,1}=B^2|_{y=0,1}=0$. Then letting $n^k$ denote the 
$k$-th
component of the outward pointing normal vector $n$ on the boundary $\partial([0,1]^2)$, we may write 
\begin{align}
    \frac{d}{dt} \int_{[0,1]^2} B^1 &= \int_{[0,1]^2} \partial_k( B^k u^1 - u^k B^1 )  = \int_{\partial[0,1]^2} (B^k u^1-u^k B^1)n^k \, . \label{eq:constant:mean}
\end{align}
When $k=2$, i.e., on the top and bottom portions of the boundary, the integrand vanishes due to the assumptions on $u^2$ and $B^2$ at $y=0,1$.  However, the integrand {also} vanishes when $k=1$ since $B^1u^1-u^1B^1\equiv 0$. Thus we deduce that the \emph{mean of $B^1$ over the square is constant in time}. 

A substantial difficulty arises in the construction of the solution to an MHD-type system in~\eqref{eq:MHD:shears}. Construction requires changing to Els\"asser variables $(u+B, u-B)$ and taking the curl of the new equations. In order to show that one can ``undo" the curl and go back to the original $(u,B)$ variables, one must show that the two Els\"asser pressures agree, or equivalently that the solution to a certain elliptic equation vanishes, cf.~Lemma~\ref{lem:equivalence} below. If the solution of the elliptic equation does not vanish, then the two Els\"asser pressures do not agree, and returning to the original variables leads to an artificial forcing term in the equation for the magnetic field. Ensuring that the solution to the elliptic equation vanishes does not seem to mesh easily with the fact that $u$ and $B$ may penetrate the boundary. Rissel and Wang \cite{risselwang} comment further on this important issue in the introduction of their paper.

These aspects of the control problem are unique to the MHD equations; indeed, consider what happens to the mean of $u^1$ in the control problem for the Euler equations.  If $u_0=(\mathcal{U},0)$ is constant, then one may construct an exact solution to the Euler equations by 
\begin{equation}\notag
    u(t,x)=g(t)e_x, \qquad p(t,x)=xg'(t) \, ,
\end{equation}
where $g(t)$ is any function satisfying $g(0)=\mathcal{U}$. So in order to drive a constant horizontal shear to zero, one may use the pressure as a forcing term to extinguish the shear.  Of course such a construction is impossible in the equation for the magnetic field in MHD, leading to the condition \eqref{eq:mean}, as well as the modified statement of the theorem with a bulk forcing term $h(t)e_x$ in case \eqref{eq:mean} is not satisfied.

Since the mean of $B$ must be constant in time, it seems plausible that \eqref{eq:mean} may be replaced with the slightly weaker condition
\begin{equation}\label{eq:mean:alt}
    \int_{[0,1]^2} \Bstart^1 \,dx\,dy = \int_{[0,1]^2} \Bend^1 \,dx\,dy \, .
\end{equation}
Such a strengthening would be optimal, and we pose it as an open question whether Theorem~\ref{th:main} can be obtained in this way; cf.~Section~\ref{ss:two} for further comments.

Extensions of our results to other domains or to higher dimensions would also be interesting. 

\subsection{Simplifications and setup}\label{ss:setup}
We claim that we can reduce the problem to
\begin{subequations}\label{eq:further:simp}
\begin{align}
     T&=1 \label{eq:further:simp:1}\\
     \|\ustart\|_{{H^r}} + \|\Bstart\|_{{H^r}} &\ll 1 \label{eq:further:simp:2} \\ 
    (\uend,\Bend)&=(0,0) \label{eq:further:simp:3}\\
    h(t) &= H'(t) \label{eq:further:simp:4} \, ,
\end{align}
\end{subequations}
where $H(t)\colon[0,1]\rightarrow \mathbb{R}$ is any smooth function satisfying
\begin{equation}\notag
    H(0) = \int_{[0,1]^2} B^1_0 \, , \qquad H(1) = 0\, , \qquad H \ll 1 \, .
\end{equation}
To see that these simplifications still imply Theorem~\ref{th:main} in full generality, first note that the MHD equations are invariant under the rescaling 
\begin{align}
     u(t,x) \rightarrow \lambda u(\lambda t, x), \qquad B(t,x)\rightarrow \lambda B(\lambda t, x), \qquad p \rightarrow \lambda^2 p(\lambda t, x)
     \, .
   \label{EQ04}
\end{align}
In the case that \eqref{eq:mean} is satisfied, we choose $H\equiv 0$.  
Then for $\lambda=\sfrac{T}{2}$, we rescale $(\ustart,\Bstart)\rightarrow \lambda(\ustart,\Bstart)$ and $(\uend,\Bend)\rightarrow \lambda(\uend,\Bend)$, and send both to $(0,0)$ in time $1$ using solutions $(\tildeustart, \tildeBstart)$ and $(\tildeuend, \tildeBend)$ to MHD, respectively. Then we reverse the direction of time
and change the signs of $(\tildeuend, \tildeBend)$, due to the scaling \eqref{EQ04},
and glue it together with $(\tildeustart, \tildeBstart)$ to produce
  \begin{equation}\notag
    (\tilde u, \tilde B)\colon[0,2]\times[0,1]^2 \rightarrow \R^3\times\R^3, \qquad (\tilde u ,\tilde B)|_{t=0}=\lambda(\ustart,\Bstart), \qquad (\tilde u ,\tilde B)|_{t=2}=\lambda(\uend,\Bend) \, .
  \end{equation}
Then defining 
\begin{equation}\notag
(u,B)(t,x)=\lambda^{-1}(\tilde u, \tilde B)(\lambda^{-1}t,x) \, ,
\end{equation}
we obtain a solution $(u,B)\colon\Omega\times[0,T]$ to \eqref{eq:mhd} satisfying \eqref{eq:data:conditions} and \eqref{eq:solution:conditions}.

In the case that \eqref{eq:mean} is \emph{not} satisfied, we may set $H$ to be a suitable non-constant function.  
Note that from \eqref{eq:constant:mean} and assuming that \eqref{eq:mean} is not satisfied, it is not possible for $H$ to be a constant function unless
\begin{equation}\notag
    \int_{[0,1]^2} B_0^1 \,dx\,dy = \int_{[0,1]^2} B_1^1 \,dx\,dy \, .
\end{equation}
In any case, proceeding as before, we obtain that $(u,B)$ solves the control problem, but with a forcing term $h(t)e_x = H'(t)e_x$ in the equation for the magnetic field. Therefore, we work under the assumptions 
\eqref{eq:further:simp:1}--\eqref{eq:further:simp:4} 
from here on.

\subsection{Outline}
The argument is structured as follows. The three steps
are addressed in Section~\ref{ss:first}, \ref{ss:two}, and~Sections~\ref{ss:three}, respectively.
\begin{enumerate}
    \item Show that the domain and the initial data $(u_0,B_0)$ may be extended to yield functions which are periodic in $x$ on a larger domain and still satisfy the appropriate divergence-free and  impermeability conditions.  This is achieved in Lemma~\ref{lem:periodic:extension}.
    Then we show that we can drive the system to a state in which the mean of $B^1(t)$ vanishes at some time $t$. This is achieved in Proposition~\ref{lem:local:one}.
    \item Show that for divergence-free vector fields on the square $[0,1]^2$ for which the mean of the first component vanishes (such as $B(t)$ after the application of the previous step), there is a divergence-free extension which vanishes on a large portion of $[-1,5]\times[0,1]$.
This is achieved in Lemma~\ref{lem:compact:truncation}.
    Then we show that compactly supported magnetic fields $B$ may be expelled from the domain $[0,1]^2$ using a strong, background, horizontal shear in $u$.  We carry out this step on the periodic domain $\mathbb{T}\times[0,1]$, where $\mathbb{T}=[0,6]$ extended periodically. This is achieved in Section~\ref{ss:two}. 
    \item Now that the magnetic field vanishes on $[0,1]^2$, the MHD on this domain reduces to the Euler equations,
and we may appeal to known control results for the Euler equations.
This is achieved in Section~\ref{ss:three}.
\end{enumerate}

%

\textbf{Acknowledgements:} 
IK was supported in part by the NSF grant DMS-1907992.
MN was supported in part by the NSF under grant DMS-1928930 while participating in a program hosted by the Mathematical Sciences Research Institute in Berkeley, California, during the spring 2021 semester.
VV was supported in part by the NSF CAREER Grant DMS-1911413.

\section{Proof of Theorem~\ref{th:main}}
\subsection{First step: extensions and local existence near background shears}\label{ss:first}

The first step consists of a lemma on extension of divergence-free vector fields and a local existence-type theorem for an MHD-type equation in the presence of a background shear.
In the remainder of this section, we denote
  \begin{equation*}
   \mathbb{T}=[0,6]   
  \end{equation*}
extended periodically.

\begin{lemma}[Extending to periodic data]\label{lem:periodic:extension}
Given an $H^r$ divergence-free vector field ${u}\colon[0,1]^2\rightarrow\mathbb{R}^2$ with $u^2|_{y=0,1}=0$, there exists ${u}_E=(u_E^1,u_E^2)\colon\mathbb{T}\times[0,1]\rightarrow\mathbb{R}^2$ 
in $H^{r}(\mathbb{T}\times[0,1])$
such that 
$$u_E|_{0\leq x \leq 1}=u\, , \qquad u_E^2|_{y=0,1}=0 \, , \qquad \int_{\mathbb{T}\times[0,1]} u_E^2 = 0 \, ; $$ in particular,
${u}_E$ is periodic in $x$ with period $6$, and ${u}_E$ satisfies the inequality
$$  \| {u}_E \|_{H^r(\mathbb{T}\times[0,1])} \lesssim \| u \|_{H^r([0,1]^2)} \, .  $$
\end{lemma}

Next, we state the local existence theorem, which is the workhorse of the paper. This proposition is stated on the set $\mathbb{T}\times[0,1]$ and demonstrates the local existence of smooth solutions near background shears.

\begin{proposition}
[Local existence near background shears]\label{lem:local:one}
Let $u_0,B_0\colon\mathbb{T}\times[0,1]\rightarrow\mathbb{R}^2$ be divergence-free vector fields with sufficiently small $H^r$ norm,
where $r\geq 3$ is an integer,
 and assume that the means of $u_0^2$ and $B_0^2$ vanish. Let $H\colon[0,1]\rightarrow\mathbb{R}$ be a smooth function depending on $B_0$, and $H_u\colon[0,1]\rightarrow\mathbb{R}$ a smooth function depending on $u_0$, as in~\eqref{eq:decomp}--\eqref{EQ08}. Then there exists a $\mathbb{T}$-periodic solution $(u,B,\nabla p)$, defined for $t\in[0,1]$, to the following MHD-type system which is close to
the background shear profiles $\ushear=H_u(t)e_x$ and $\Bshear=H(t)e_x$ and solves
\begin{subequations}
\label{eq:MHD:shears}
\begin{align}
\partial_t u + u\cdot\nabla u + \nabla p &= B\cdot \nabla B \\
\partial_t B + u\cdot\nabla B - B\cdot\nabla u &= H'(t)e_x  \\
\Div u = \Div B &= 0 \\
u^2|_{y=0,1} = B^2|_{y=0,1} &= 0 \\
\int_{\mathbb{T}\times[0,1]} u^2(t) \equiv \int_{\mathbb{T}\times[0,1]} B^2(t) &\equiv 0 \\ 
 u|_{t=0}&=u_{0} \\ 
 B|_{t=0}&=B_0 \\
 \int_{[0,1]^2} B^1|_{t=1}&= 0
\,. \label{eq:mean:zero:B}
\end{align}
\end{subequations}
Furthermore, if $\int_{[0,1]^2} B_0^1=0$, then we may take $H\equiv 0$.
\end{proposition}

For the precise quantification of how close the solution needs to be to the background shear, cf.~Lemma~\ref{lem:one}.


\begin{proof}[Proof of Lemma~\ref{lem:periodic:extension}]
Introduce the stream function
  \begin{equation}
   \psi(x,y)
   = 
   -
   \int_{\ell}
       u^{\perp}(\tilde x, \tilde y) \cdot (d\tilde x,d\tilde y)
   \,,
   \label{EQ07}
  \end{equation}
where $\ell$ denotes a sufficiently regular path from $(0,0)$ to $(x,y)$, which satisfies
$\nabla^{\perp}\psi=u$. Note
that the integral is independent of a chosen path since
$\Div u=0$. Clearly,
the condition $u^2|_{y=0,1}=0$ implies that
$\psi|_{y=0}=0$ and $\psi$ is constant on the upper boundary $\{y=1\}$, whose value
we denote by
$\cC_\psi$.
Since it is needed below, note that, in particular,
  \begin{equation}
   \int_{0}^{1} u^1(x,y) \,dy = \cC_{\psi} 
   \label{eq:psi:constant}
   ,
  \end{equation}
i.e., the integral $\int_{0}^{1} u^1(x,y) \,dy$ is independent of $x$.
Let $\tilde \psi\colon[-1,2]\times[0,1]\rightarrow\mathbb{R}^2$ be an $H^r$ Sobolev extension/reflection operator over $x=0$ and $x=1$. By the explicit formula 
for such an extension operator (cf.~Evans~\cite{evans}),
$\tilde\psi$ is still constant on $\{y=0\}$ and $\{y=1\}$ for $x\in[-1,2]$.  Therefore $\tilde u =\nabla^\perp \tilde \psi$ satisfies $\tilde u^2|_{y=0,1}=0$ on the extended set. 

Now, let $\theta\colon[-1,2]\rightarrow[0,1]$ be a function depending only on $x\in[-1,2]$ which satisfies 
\begin{equation}\notag
    \theta(x)=0 \quad \textnormal{if} \quad x\in[-\sfrac{1}{2},\sfrac{3}{2}]
\end{equation}
and
\begin{equation}\notag
    \theta(x)=1 \quad \textnormal{if} \quad x\in[-1,-\sfrac{3}{4}] \cup [\sfrac{7}{4},2]
    \,.
\end{equation}
Define 
$\psi_{\mathbb{T}}\colon[-1,2]\times[0,1]\rightarrow\mathbb{R}$ by 
$$ \psi_{\mathbb{T}} = \cC_\psi \theta(x) y + (1-\theta(x)) \tilde \psi(x,y) \, .  $$
Then due to $\tilde\psi(x,0)=0$ and $\tilde\psi(x,1)=\cC_\psi$, the function $\psi_{\mathbb{T}}$ is constant on $y=0,1$, i.e., on the entire upper and lower boundaries of the extended domain $[-1,2]\times[0,1]$.  
Since $\psi_{\mathbb{T}}$ is uniformly equal to $\cC_\psi y$ for $x\in[-1,-\sfrac{3}{4}]\cup[\sfrac{7}{4},2]$
due to the properties of $\theta$,
we may extend $\tilde\psi_{\mathbb{T}}$ 
periodically in $x$ with period~$6$, i.e., we may assume that it is defined on $\mathbb{T}$
with preserved smoothness properties.
Then define $u_E = \nabla^\perp\tilde \psi_{\mathbb{T}}$. 
Moreover,
$$  \int_{\mathbb{T}\times[0,1]} u_E^2 =  \int_{[-1,5]\times[0,1]} \partial_x  \tilde \psi_{\mathbb{T}} = 0 \, , $$
by $x$-periodicity of $\tilde \psi_{\mathbb{T}}$,
concluding the proof.
\end{proof}

Now, we turn to the proof of Proposition~\ref{lem:local:one}.
We look for a solution of this system which satisfies
\begin{subequations}\label{eq:decomp}
\begin{align}
    u &= \tilde u + H_u(t)e_x = \tilde u + \ushear \\
    H_u(0)&= \int_{[0,1]^2}u_0^1, \quad H_u(1)=0 \\
    \int_{[0,1]^2} \tilde u(t) &= 0 \qquad \forall t\in[0,1]
\end{align}
\end{subequations}
and
\begin{subequations}\label{EQ08}
\begin{align}
    B &= \tilde B + H(t)e_x = \tilde B + \Bshear \\
    H(0)&= \int_{[0,1]^2}B_0^1, \quad H(1)= 0 \\
    \int_{[0,1]^2} \tilde B(t) &= 0 \qquad \forall t\in[0,1] 
\, .
\end{align}
\end{subequations}
With this ansatz in mind and noting that  $\nabla \ushear=\nabla
\Bshear=0$, and that $\partial_t \ushear$ is equal to a pressure
$\partial_x\pshear$ which is periodic (although $\pshear$ itself is \emph{not} periodic), \eqref{eq:MHD:shears} now reads
\begin{subequations}
\label{eq:MHD:shear:decomp}
\begin{align}
\partial_t \tilde u + ( \tilde u + \ushear ) \cdot\nabla\tilde u + \nabla q &= (\tilde B + \Bshear) \cdot \nabla \tilde B \\
\partial_t \tilde B + (\tilde u + \ushear) \cdot \nabla \tilde B - (\tilde B + \Bshear) \cdot \nabla \tilde u &= 0 \label{eq:B:pert:one} \\
\Div \tilde u = \Div \tilde B &= 0 \\
\tilde u^2|_{y=0,1} = \tilde B^2|_{y=0,1} &= 0 \\
\int_{\mathbb{T}\times[0,1]} \tilde u^2(t) \equiv \int_{\mathbb{T}\times[0,1]} \tilde B^2(t) &\equiv 0 \\ 
 \tilde u|_{t=0}&=u_{0} - \ushear|_{t=0} \\ 
 \int_{[0,1]^2} \tilde u^1(t) &\equiv 0 \\ 
 \tilde B|_{t=0}&=B_0 - \Bshear|_{t=0} \\
 \int_{[0,1]^2} \tilde B^1(t)&\equiv 0
 \,.
\end{align}
\end{subequations}
Observe that $\int_{[0,1]^2}\tilde u^{1}(t)=0$ is equivalent to $\int_{{\mathbb T}\times[0,1]}\tilde u^{1}(t)=0$ by
\eqref{eq:psi:constant} resulting from
the divergence-free condition.
Similarly, $\int_{[0,1]^2}\tilde B^{1}(t)=0$ is equivalent to $\int_{{\mathbb T}\times[0,1]}\tilde B^{1}(t)=0$.
We shall prove that one can solve this system for $\tilde u$, $\tilde B$, and $q$ which are $\mathbb{T}$-periodic in $x$ and that $q\colon\mathbb{T}\times[0,1]\rightarrow\mathbb{R}$ solves an elliptic problem that enforces  $\Div \tilde u=0$ and $\tilde u^2|_{y=0,1} =0 $: 
\begin{subequations}
\notag
\begin{align}
\Delta q &= \Div(-\tilde u\cdot\nabla\tilde u - \ushear\cdot\nabla\tilde u + (\tilde B + \Bshear)\cdot \nabla \tilde B ) \\
\partial_y q|_{y=0,1} &= 0 \\
\int_{\mathbb{T}\times[0,1]} q &=0 \, .
\end{align}
\end{subequations}
Note that the MHD system exhibits a loss of derivatives; thus in order
to solve this system,
we need to switch to the Els\"asser variables
\begin{equation}\notag
z_1 = \tilde u + \tilde B, \qquad z_2 = \tilde u - \tilde B \, .  
\end{equation}
In these variables, the equations in \eqref{eq:MHD:shear:decomp} become
\begin{subequations}
\label{eq:MHD:shear:elsasser}
\begin{align}
\partial_t z_1 + ( z_2 + \ushear - \Bshear ) \cdot\nabla z_1 + \nabla q &= 0 \\
\partial_t z_2 + ( z_1 + \ushear + \Bshear ) \cdot\nabla z_2 + \nabla q &= 0 \\
\Div z_1 = \Div z_2 &= 0 \\
z_1^2|_{y=0,1} = z_2^2|_{y=0,1} &= 0 \label{eq:elsasser:normal} \\
\int_{\mathbb{T}\times[0,1]} z_1(t) \equiv \int_{\mathbb{T}\times[0,1]} z_2(t) &\equiv 0 \label{eq:elsasser:mean} \\ 
 z_1|_{t=0}&= u_{0} - \ushear|_{t=0} + B_0 - \Bshear|_{t=0} \\ 
  z_2|_{t=0}&= u_{0} - \ushear|_{t=0} - B_0 + \Bshear|_{t=0} \, .
\end{align}
\end{subequations}
Note that the conditions on the means of $\tilde u^1$, $\tilde u^2$, $\tilde B^1$, and $\tilde B^2$ have been consolidated into \eqref{eq:elsasser:mean}, asserting that the means of \emph{both} components of $z_1$ and $z_2$ vanish. Taking the curl of the first two equations in \eqref{eq:MHD:shear:elsasser} yields
\begin{subequations}
\label{eq:MHD:shear:elsasser:curl}
\begin{align}
\omega_1 &= \nabla^\perp \cdot z_1 \\
\omega_2 &= \nabla^\perp \cdot z_2 \\
\partial_t \omega_1 + ( z_2 + \ushear - \Bshear ) \cdot\nabla \omega_1 &= -\partial_k z_1^\ell \epsilon_{\ell j} \partial_j z_2^k \label{eq:curl:one} \\
\partial_t \omega_2 + ( z_1 + \ushear + \Bshear ) \cdot\nabla \omega_2 &= -\partial_k z_2^\ell \epsilon_{\ell j} \partial_j z_1^k \, . \label{eq:curl:two}
\end{align}
\end{subequations}
Lemma~\ref{lem:equivalence} below shows that if we have solved this ``vorticity-Els\"asser-MHD" system, where we have substituted \eqref{eq:MHD:shear:elsasser:curl} for the first two equations in \eqref{eq:MHD:shear:elsasser}, then in fact we have solved \eqref{eq:MHD:shear:decomp}. We first need the following De~Rham-type result.

\begin{lemma} [Periodic De~Rham's theorem]
\label{L01}
Assume that $v\in L^{2}_{\text{loc}}({\mathbb R}\times[0,1])$ is $L$-periodic in the $x$ variable, 
where $L>0$, 
and suppose that it satisfies
$\nabla^{\perp}\cdot v = 0$ 
and $\int_{[0,L]\times[0,1]} v^{1} = 0$.
Then there exists a function
$q\in H^{1}_{\text{loc}}({\mathbb R}\times[0,1])$, which is $L$-periodic in the $x$ variable,
and satisfies
  \begin{equation}
   v=\nabla q   
   \label{EQ01}
  \end{equation}
on ${\mathbb R}\times (0,1)$.
\end{lemma}

\begin{proof}[Proof of Lemma~\ref{L01}]
By the classical De~Rham's theorem \cite[Proposition I.1.1]{Temam}, there exists 
a distribution $q\in {\mathcal D}'({\mathbb R}\times(0,1))$ such that
\eqref{EQ01} holds.
Using \cite[Proposition I.1.2(i)]{Temam},
we have
$q\in H^{1}_{\text {loc}}({\mathbb R}\times[0,1])$, so it only remains
to establish periodicity.
By the periodicity of $v$, we have
$\nabla (q(x+L,y)-q(x,y))=0$, 
for $(x,y)\in{\mathbb R}\times(0,1)$,
which implies that
$q(x+L,y)-q(x,y)=a$, 
for all $(x,y)\in {\mathbb R}\times(0,1)$,
where $a\in {\mathbb R}$ is a constant.
Since $ 0 =\int_{[0,L]\times[0,1]} v^{1}= \int_{[0,L]\times[0,1]} \partial_{1} q = \int_{[0,1]} (q(L,y) - q(0,y)) = a$, we get $a=0$, implying the $L$-periodicity of~$q$.
\end{proof}

We note in passing that any smooth vector field 
$v\colon\mathbb{T}\times[0,1]\rightarrow\mathbb{R}^2$ which satisfies $\int_{\mathbb{T}\times[0,1]} v^1 = 0$ allows a unique $L^2(\mathbb{T}\times[0,1])$-orthogonal decomposition of the form
\begin{equation}\notag
    v = \nabla p + \nabla^\perp q \, , \qquad \partial_y p|_{y=0,1} = v^2|_{y=0,1} \, , \qquad q|_{y=0,1}=0 \, , \qquad \int_{\mathbb{T}\times[0,1]} p = \int_{\mathbb{T}\times[0,1]} q = 0 \, ,
\end{equation}
where $p,q\colon\mathbb{T}\times[0,1]\rightarrow\mathbb{R}$ are smooth and periodic. We construct $q$ 
as the solution to the elliptic problem
\begin{subequations}\notag
\begin{align}
    \Delta q &= \nabla^\perp\cdot v \\
    q|_{y=0,1} &= 0 \\
    \int_{\mathbb{T}\times[0,1]} q &= 0 \, .
\end{align}
\end{subequations}
Now, considering $v-\nabla^\perp q$, we have
$\nabla^\perp\cdot(v - \nabla^\perp q)=0$
and 
$\int_{\mathbb{T}\times[0,1]} (v^{1}+\partial_{2} q)=0$.
Applying Lemma~\ref{L01} to $v - \nabla^\perp q$, we may write it as the gradient of a periodic function $p$, which without loss of generality may be taken to have zero mean. The $L^2$-orthogonality is immediate from integration by parts, the fact $ q|_{y=0,1}=0$ by construction, and the periodicity in $x$ of $v$, $p$, and $q$. Uniqueness follows from the construction, in particular the imposition of the mean-zero conditions.

\begin{lemma}[Solving vorticity-Els\"asser MHD]\label{lem:equivalence}
Solving \eqref{eq:MHD:shear:elsasser} but with \eqref{eq:MHD:shear:elsasser:curl} taking the place of the first two equations in \eqref{eq:MHD:shear:elsasser} is equivalent to solving \eqref{eq:MHD:shear:decomp}.  Consequently, solving either provides a solution to \eqref{eq:MHD:shears}.
\end{lemma}
\begin{proof}[Proof of Lemma~\ref{lem:equivalence}]
Assume that we have a solution of \eqref{eq:MHD:shear:elsasser:curl}.
It is easy to check that
  \begin{align}\notag
     &
     \nabla^\perp \cdot \left( \partial_t z_1 + ( z_2+\ushear-\Bshear )\cdot\nabla z_1 \right)
     \\&\indeq
     =
     \partial_t \omega_1 + ( z_2 + \ushear - \Bshear ) \cdot\nabla \omega_1 +\partial_k z_1^\ell \varepsilon_{\ell j} \partial_j z_2^k 
     = 0
    \,.
   \notag
  \end{align}
In order to apply Lemma~\ref{L01}, 
we need to verify that the integral 
over $\mathbb{T}\times[0,1]$
of the first component of 
$\partial_t z_1 + ( z_2+\ushear-\Bshear )\cdot\nabla z_1$ vanishes.
For the first term, $\partial_{t}z_1$, this is clear, while for the second we have
  \begin{equation}
    \int_{{\mathbb T}\times[0,1]}
      ( z_2+\ushear-\Bshear )\cdot\nabla z_1^{1}
    = 
    \int_{{\mathbb T}\times[0,1]}
      \partial_{i}
       (
        (z_2^{i}+\ushear^{i}-\Bshear^{i} )z_1^{1}
       )
    = 0
     \,,
   \notag
  \end{equation}
where in last equality we separately integrate for $i=1$ and $i=2$. When $i=1$, we use periodicity, while
when $i=2$ it is important that the
expression inside the parentheses vanishes for $y=0$ and $y=1$.
By Lemma~\ref{L01}, there exists a $\mathbb{T}$-periodic function $q_1$ such that 
$$  \partial_t z_1 + (z_2 + \ushear - \Bshear ) \cdot \nabla z_1 = - \nabla q_1 \, ,  $$
where $q_1\colon\mathbb{T}\times[0,1]\rightarrow\mathbb{R}$ solves
\begin{subequations}\notag
\begin{align}
    -\Delta q_1 &= \partial_k \bigl( (z_2+\ushear-\Bshear)^\ell \partial_\ell z_1^k \bigr) \\
    \partial_2 q_1|_{y=0,1} &= 0 \\
    \int_{\mathbb{T}\times[0,1]} q_1 &= 0  \, .
\end{align}
\end{subequations}
We similarly have 
$$  \partial_t z_2 + ( z_1+\ushear-\Bshear )\cdot\nabla z_2 = - \nabla q_2 \, , $$
where
\begin{subequations}\notag
\begin{align}
    -\Delta q_2 &= \partial_k \bigl( (z_1+\ushear+\Bshear)^\ell \partial_\ell z_2^k \bigr) \\
    \partial_2 q_2|_{y=0,1} &= 0 \\
    \int_{\mathbb{T}\times[0,1]} q_2 &= 0 \, .
\end{align}
\end{subequations}
We find that $q_1-q_2$ is $\mathbb{T}$-periodic in $x$ and solves
\begin{subequations}\notag
\begin{align}
    -\Delta (q_1-q_2) &= 0 \\
    \partial_2 (q_1-q_2)|_{y=0,1} &= 0 \\
    \int_{\mathbb{T}\times[0,1]} (q_1-q_2) &= 0 \, ,
\end{align}
\end{subequations}
from where $q_1=q_2$ and we have a solution to \eqref{eq:MHD:shear:elsasser};
to obtain that $q_1-q_2$ is harmonic, 
we write
  \begin{align}
   \begin{split}
     -\Delta(q_1-q_2)        
      &= -\partial_k\partial_{\ell} \bigl( (z_2+\ushear-\Bshear)^\ell z_1^k \bigr)
        + \partial_k\partial_{\ell} \bigl( (z_1+\ushear+\Bshear)^\ell  z_2^k \bigr)
     \\&
      = -\partial_k\partial_{\ell} \bigl( (\ushear-\Bshear)^\ell z_1^k \bigr)
        + \partial_k\partial_{\ell} \bigl( (\ushear+\Bshear)^\ell  z_2^k \bigr)
      \\&
      =
       -\partial_{\ell} \bigl( (\ushear-\Bshear)^\ell \partial_k z_1^k \bigr)
        + \partial_{\ell} \bigl( (\ushear+\Bshear)^\ell  \partial_kz_2^k \bigr)
      = 0
     \,,
   \end{split}
   \notag
  \end{align}
where we used $\nabla \ushear=\nabla \Bshear=0$ in the third equality and the divergence-free condition in the last.
Reconstructing the equations for $\tilde u$ and $\tilde B$ from $z_1$ and $z_2$ as usual and using that $q_1=q_2$ then shows that we have a solution to \eqref{eq:MHD:shear:decomp}. To conclude the proof, we must demonstrate the other direction of the equivalence, but this only 
amounts to taking the curl of the first two equations in \eqref{eq:MHD:shear:elsasser}.  
\end{proof}

Returning to the proof of Proposition~\ref{lem:local:one}, we will be done if we can set up and solve the fixed point iteration:
\begin{subequations}
\label{eq:shear:fixed:point}
\begin{align}
\omega\onen &= \nabla^\perp \cdot z\onen \label{eq:shear:first} \\
\omega\twon &= \nabla^\perp \cdot z\twon \\
\partial_t \omega\onen + ( z\twonminus + \ushear - \Bshear ) \cdot\nabla \omega\onen &= -\partial_k z\onenminus^\ell \varepsilon_{\ell j} \partial_j z\twonminus^k \label{eq:shear:one} \\
\partial_t \omega\twon + ( z\onenminus + \ushear + \Bshear ) \cdot\nabla \omega\twon &= -\partial_k z\twonminus^\ell \varepsilon_{\ell j} \partial_j z\onenminus^k \label{eq:shear:two} \\
\Div z\onen = \Div z\twon &= 0 \label{eq:shear:three} \\
z\onen^2|_{y=0,1} = z\twon^2|_{y=0,1} &= 0 \label{eq:shear:four}\\
 z\onen|_{t=0}&= u_{0} - \ushear|_{t=0} + B_0 - \Bshear|_{t=0} \\ 
  z\twon|_{t=0}&= u_{0} - \ushear|_{t=0} - B_0 + \Bshear|_{t=0} \\
 \int_{\mathbb{T}\times[0,1]} z\onen(t)\equiv\int_{\mathbb{T}\times[0,1]} z\twon(t)&\equiv 0 \, . \label{eq:shear:last}
\end{align}
\end{subequations}
We split the proof into the following three lemmas.

\begin{lemma}[Constructing the iterates]\label{lem:construct}
There exists a sequence $\{(z\onen,z\twon)\}_{n=0}^\infty$ which is well-defined and satisfies the equations in \eqref{eq:shear:fixed:point}.
\end{lemma}

\begin{lemma}[Uniform bound in high regularity]\label{lem:one}
The sequence $\{(z\onen,z\twon)\}_{n=0}^\infty$ belongs to the ball\\ $B_{C\varepsilon}(0)\subset C([0,1];{H^r}(\mathbb{T}\times[0,1]))$ provided the initial data for $z\onen$ and $z\twon$ from \eqref{eq:shear:fixed:point} are smaller than ${\varepsilon}$ in $C([0,1];{H^r}(\mathbb{T}\times[0,1]))$, where $\varepsilon$ is sufficiently small and $C$ is a constant.
\end{lemma}

The smallness of data can always be ensured by a sufficiently powerful 
$\varepsilon$-dependent rescaling of the original problem, cf.~\eqref{eq:further:simp}.

\begin{lemma}[Contraction in $H^1$]\label{lem:two}
The sequence $\{(z\onen,z\twon)\}_{n=1}^\infty$ satisfies the contraction inequality
\begin{align}\label{eq:contraction}
    &\| z\onen - z\onenplus \|_{L^\infty([0,1];H^1(\mathbb{T}\times[0,1]))} + \| z\twon - z\twonplus \|_{L^\infty([0,1];H^1(\mathbb{T}\times[0,1]))} \notag\\
    &\indeq \leq \frac{1}{2} ( \| z\onen - z\onenminus \|_{L^\infty([0,1];H^1(\mathbb{T}\times[0,1]))} + \| z\twon - z\twonminus \|_{L^\infty([0,1];H^1(\mathbb{T}\times[0,1]))} )
\end{align}
provided that $\varepsilon$, $z_{1,0}$, and $z_{2,0}$ are sufficiently small.
\end{lemma}

\begin{proof}[Proof of Proposition~\ref{lem:local:one}]
Assuming the conclusions of the preceding three lemmas, the contraction mapping principle provides a unique limit point of the sequence $(z\onen, z\twon)$, that is, we have a unique solution to the system of equations \eqref{eq:MHD:shear:elsasser}, and thus \eqref{eq:MHD:shears}, concluding the proof of Proposition~\ref{lem:local:one}. 
\end{proof}

Thus it remains to prove Lemmas~\ref{lem:construct}--\ref{lem:two}. To prove Lemma~\ref{lem:construct}, we first need to state an existence and uniqueness theorem for a periodic div-curl problem.

\begin{lemma}[Periodic div-curl problem]
\label{L02}
Consider the system
  \begin{align}
    \Div z = 0 \comma
    \curl z = \omega \comma
    z^{2}|_{y=0,1} = 0 \comma
    \int_{\mathbb{T}\times[0,1]} z = 0
    \,.
   \label{EQ03}
  \end{align}
For every $\omega\in H^{s}(\mathbb{T}\times[0,1])$, where $s\geq1$, there
exists a unique solution $z\in H^{s+1}(\mathbb{T}\times[0,1])$.
Also, the mapping $\omega\mapsto z$ is continuous
from $H^{s}(\mathbb{T}\times[0,1])$ to $H^{s+1}(\mathbb{T}\times[0,1])$.
\end{lemma}

\begin{proof}[Proof of Lemma~\ref{L02}]
To obtain the existence, first solve
$\Delta\psi=\omega$ for $\psi\colon\mathbb{T}\times[0,1]\rightarrow\mathbb{R}$ 
with the boundary conditions $\psi|_{y=0,1}=0$ on top and bottom 
and periodic boundary conditions in $x$. 
Then set $z=\nabla^\perp\psi$. It is easy to check that $z$ 
satisfies the first three conditions in \eqref{EQ03}.
For the fourth condition in \eqref{EQ03}, we have
$\int_{\mathbb{T}\times[0,1]} z^{2}=\int_{\mathbb{T}\times[0,1]} \partial_{1}\psi = 0$ since $\psi$ is periodic in $x$. Also, $ \int_{\mathbb{T}\times[0,1]} z^{1}=-\int_{\mathbb{T}\times[0,1]} \partial_{2}\psi = \int_{\mathbb{T}} (\psi(x,0) - \psi(x,1)) = 0$, since $\psi|_{y=0,1}=0$.

For uniqueness, assume that $z$ is periodic and satisfies \eqref{EQ03} with $\omega=0$.
By $\curl z=0$ 
and $\int_{\mathbb{T}\times[0,1]} z^{1} = 0$, there exists a periodic function
$\phi$ such that
$z=\nabla \phi$. The divergence condition and
$    z^{2}|_{y=0,1}=0$ then imply that $\phi$ solves the homogeneous Neumann problem
$\Delta\phi=0$ and $\partial_{2}\phi|_{y=0,1}=0$. Thus $\phi$ is constant, from where
$z=0$.
\end{proof}

Note that the proof of uniqueness does not use
$\int_{\mathbb{T}\times[0,1]}z^{2}=0$ showing that 
this is a consequence of $\Div z=0$ and $\int_{\mathbb{T}\times[0,1]} z^{1} = 0$.


\begin{proof}[Proof of Lemma~\ref{lem:construct}]
We begin by defining the time-independent first iterates
\begin{equation}\label{eq:first:iterates}
z_{1,0} = u_0 - \ushear |_{t=0} + B_0 - \Bshear|_{t=0} \, , \qquad z_{2,0} = u_0 - \ushear |_{t=0} - B_0 + \Bshear|_{t=0} \, . 
\end{equation}
From \eqref{eq:decomp}--\eqref{EQ08} we have that $z_{1,0}^1$ and $z_{2,0}^1$ have vanishing averages over $\mathbb{T}\times[0,1]$ for all times, as desired in \eqref{eq:shear:last}. From the assumptions of Proposition~\ref{lem:local:one}, we have that the averages of $z_{1,0}^2$ and $z_{2,0}^2$ also vanish for all times.  Both first iterates are also clearly divergence free and satisfy \eqref{eq:shear:four}.  

Now assume that the pair $(z\onenminus,z\twonminus)$ is given for $n\geq 1$ and satisfies \eqref{eq:shear:first}--\eqref{eq:shear:last}. To construct $(z\onen, z\twon)$, we first solve \eqref{eq:shear:one} and \eqref{eq:shear:two} for $\omega\onen$ and $\omega\twon$, respectively, using the method of characteristics on the set $\mathbb{T}\times[0,1]$. This is possible because the velocity fields and forcing terms for both equations are periodic in $x$, and the velocity fields do not penetrate the boundaries at $y=0$ and $y=1$. 
We then solve the $\Div$-$\curl$ problem
\begin{subequations}\notag
\begin{align}
    \Div z_{i,n}(t) &= 0 \\
    \curl z_{i,n}(t) &= \omega_{i,n}(t) \\
    z_{i,n}^2(t)|_{y=0,1} &= 0 \\
    \int_{\mathbb{T}\times[0,1]} z_{i,n}(t) &= 0
    \,,
\end{align}
\end{subequations}
using Lemma~\ref{L02}.
for $z_{i,n}(t)\colon\mathbb{T}\times[0,1]\rightarrow\mathbb{R}^2$ with $i=1,2$ and $t\in[0,1]$, obtaining the estimate
\begin{align}
    \| z_{i,n} \|_{H^{r+1}(\mathbb{T}\times[0,1])} 
     &\leq C \| \omega_{i,n} \|_{H^r(\mathbb{T}\times[0,1])} 
   \notag
\end{align}
Finally, in order to obtain $z_{i,n}(0)=z_{0}$, we simply use
$\curl z_{i,n}(0)=\curl z_0$, the continuity of $\curl z_{i,n}$ in $t$,
and the continuity of the map $\omega\mapsto z$ in Lemma~\ref{L02}.
\end{proof}

\begin{proof}[Proof of Lemma~\ref{lem:one}]
The proof proceeds by induction on $n$ and a standard energy/Gr\"onwall argument for small data. The estimates for $n=0$ follow from the fact that we have defined the first iterates to be time-independent, cf.~\eqref{eq:first:iterates}. We now assume that the bounds have been shown for $(z\onenminus,z\twonminus)$ for $n\geq 1$ and show that the same bounds hold for $(z\onen,z\twon)$.

We first work towards $L^2$ bounds on $\omega\onen$ and $\omega\twon$.  Multiplying \eqref{eq:shear:one} by $\omega\onen$ and \eqref{eq:shear:two} by $\omega\twon$, integrating over $\mathbb{T}\times[0,1]$, and using that $z\onenminus$, $z\twonminus$, $\ushear$, $\Bshear$, $\omega\onen$, and $\omega\twon$ are periodic in $x$ and having vanishing second component at $y=0,1$, we obtain the energy inequalities
\begin{align}
    \frac{1}{2}\frac{d}{dt} \| \omega_{i,n} \|_{L^2(\mathbb{T}\times[0,1])}^2 
       &\leq 
      8 
     \| | \nabla z\onenminus | | \nabla z\twonminus | \|_{L^2(\mathbb{T}\times[0,1])}  
    \| \omega_{i,n} \|_{L^2(\mathbb{T}\times[0,1])} \notag\\
    &\leq 8 C^2 \varepsilon^2 \| \omega_{i,n} \|_{L^2(\mathbb{T}\times[0,1])} \, . \notag
\end{align}
To achieve the second inequality, we have used that {$H^s(\mathbb{T}\times[0,1])$ is an algebra when $s\geq 3$} and the inductive assumption on $z\onenminus$ and $z\twonminus$.  Integrating from $0$ to $t$ for $t\leq 1$, we obtain that for $t\in[0,1]$,
\begin{align}
    \| \omega_{i,n} (t) \|_{L^2(\mathbb{T}\times[0,1])} &\leq \frac{\varepsilon}{4} + 8 C^2\varepsilon^2 \leq \varepsilon \, , \notag
\end{align}
assuming $\varepsilon$ is sufficiently small and the assumption from Lemma~\ref{lem:one} on the size of the initial data.

Next, applying $\nabla$ to the equations \eqref{eq:shear:one} and \eqref{eq:shear:two}, integrating over $\mathbb{T}\times[0,1]$, using the same properties as before, and setting $i'=1$ if $i=2$ and $i'=2$ if $i=1$, we obtain the energy inequalities
\begin{align}
    \frac{1}{2}\frac{d}{dt} \| \nabla \omega_{i,n} \|_{L^2(\mathbb{T}\times[0,1])}^2 
        &\leq 
        8 \| \nabla( \nabla z\onenminus \tensor \nabla z\twonminus  ) \|_{L^2(\mathbb{T}\times[0,1])} \| \nabla \omega_{i,n} \|_{L^2(\mathbb{T}\times[0,1])} \notag\\
    &\qquad \qquad + \| \nabla z_{i',n-1} \|_{L^2(\mathbb{T}\times[0,1])} \| \nabla \omega_{i,n} \|^2_{L^2(\mathbb{T}\times[0,1])} \notag\\
    &\leq 8 C^2\varepsilon^2 \| \nabla \omega_{i,n} \|_{L^2(\mathbb{T}\times[0,1])} + C\varepsilon \| \nabla \omega_{i,n} \|_{L^2(\mathbb{T}\times[0,1])}^2 \, . \notag
\end{align}
To achieve the second inequality, we have again used that {$H^s(\mathbb{T}\times[0,1])$ is an algebra when $s\geq 3$} and the inductive assumptions on $z\onenminus$ and $z\twonminus$. This implies that
\begin{equation}\notag
    \| \nabla \omega_{i,n}(t) \|_{L^2(\mathbb{T}\times[0,1])} \leq \frac{\varepsilon}{4} + 8 C^2 \varepsilon^2 + \int_0^t C \varepsilon  \| \nabla \omega_{i,n}(s) \|_{L^2(\mathbb{T}\times[0,1])} \, ds \, ,
\end{equation}
and so from the integral form of Gr\"onwall's inequality, we obtain 
\begin{equation}\notag
    \| \nabla \omega_{i,n}(t) \|_{L^2(\mathbb{T}\times[0,1])} \leq \left( \frac{\varepsilon}{4} + 8 C^2 \varepsilon^2 \right) \exp\left(\int_0^1 C \varepsilon \, ds \right) \leq \varepsilon
\end{equation}
if $\varepsilon$ is chosen sufficiently small. Utilizing the elliptic estimates in Lemma~\ref{L02} and employing similar arguments but with higher-order spatial derivative concludes the proof.
\end{proof}

\begin{proof}[Proof of Lemma~\ref{lem:two}]
We set $\omega_1:=\omega\onenplus-\omega\onen$. 
From \eqref{eq:shear:one}, we find that $\omega_1$ satisfies the equation
\begin{align}
    &\partial_t\omega_1 + (z\twon + \ushear - \Bshear ) \cdot \nabla \omega_1 
   \notag\\&\indeq
    = ( z\twonminus - z\twon ) \cdot \nabla \omega\onen \notag
     - \partial_k z\onen^\ell \varepsilon_{\ell j} \partial_j z\twon^k + \partial_k z\onenminus^\ell \varepsilon_{\ell j} \partial_j z\twonminus^k \notag\\
    &\indeq= ( z\twonminus-z\twon ) \cdot \nabla \omega\onen - ( \partial_k z\onen^\ell - \partial_k z\onenminus^\ell ) \varepsilon_{\ell j} \partial_j z\twon^k \notag
    - \partial_k z\onenminus^\ell \varepsilon_{\ell j} ( \partial_j z\twon^k - \partial_j z\twonminus^k ) \, , \notag
\end{align}
with $\omega_1|_{t=0}=0$.  Multiplying by $\omega_1$, we find that 
\begin{align}
    \frac{1}{2}\frac{d}{dt} \| \omega_1 \|_{L^2(\mathbb{T}\times[0,1])}^2 &\lesssim \| z\twonminus - z\twon \|_{L^2(\mathbb{T}\times[0,1])} \| \omega_1 \nabla \omega\onen \|_{L^2(\mathbb{T}\times[0,1])} \notag\\
    &\quad + \| \nabla( z\onenminus - z\onen) \|_{L^2(\mathbb{T}\times[0,1])} \| \nabla z\twon \omega_1 \|_{L^2(\mathbb{T}\times[0,1])} \notag\\
    &\quad + \| \nabla( z\twonminus - z\twon) \|_{L^2(\mathbb{T}\times[0,1])} \| \nabla z\onenminus \omega_1 \|_{L^2(\mathbb{T}\times[0,1])} \notag\\
    &\lesssim \varepsilon \| \omega_1 \|_{L^2(\mathbb{T}\times[0,1])} \bigl( \| z\onenminus - z\onen \|_{H^1(\mathbb{T}\times[0,1])} + \| z\twonminus - z\twon \|_{H^1(\mathbb{T}\times[0,1])} \bigr) \, . \notag
\end{align}
Using a Gr\"onwall argument and choosing $\varepsilon\ll 1$ sufficiently small to absorb any constants, we deduce that
\begin{align}
    &\| \omega\onen - \omega\onenplus \|_{L^\infty([0,1];L^2(\mathbb{T}\times[0,1]))} \notag\\
    &\indeq \leq \frac{1}{4C} \bigl( \| z\onen - z\onenminus \|_{L^\infty([0,1];H^1(\mathbb{T}\times[0,1]))} + \| z\twon - z\twonminus \|_{L^\infty([0,1];H^1(\mathbb{T}\times[0,1]))} \bigr) 
   \label{EQ06}
\, ,
\end{align}
where $C$ is sufficiently large.
Utilizing Lemma~\ref{L02}, with $s=1$ while assuming the constant $C$ in \eqref{EQ06} is sufficiently large, we get
\begin{align}
    &\| \omega\onen - \omega\onenplus \|_{L^\infty([0,1];L^2(\mathbb{T}\times[0,1]))} \notag\\
    &\indeq \leq \frac{1}{4} 
      \bigl( \| \omega\onen - \omega\onenplus \|_{L^\infty([0,1];L^2(\mathbb{T}\times[0,1]))} 
              + \| \omega\twon - \omega\twonplus \|_{L^\infty([0,1];L^2(\mathbb{T}\times[0,1]))} 
      \bigr) 
   \notag
\, .
\end{align}
Making the analogous estimate for $z\twon - z\twonplus$, and summing concludes the proof of \eqref{eq:contraction}.
\end{proof}

The first step is now as follows. The initial data $u_0$ and $B_0$ are extended using Lemma~\ref{lem:periodic:extension} to data which are divergence-free and $\mathbb{T}$-periodic with the means of $u^2$ and $B^2$ vanishing over $\mathbb{T}\times[0,1]$. Choosing suitable $\ushear=H_u(t)e_x$ and
$\Bshear=H(t)e_x$, we can drive the system to the state
such that the means of
$u$ and $B$ over $[0,6]\times[0,1]$  vanish at time $t=1$.

\subsection{Second step: expelling the magnetic field}\label{ss:two}
Recall that in the first step, we solved \eqref{eq:MHD:shear:decomp} using the ansatz \eqref{eq:decomp}--\eqref{EQ08}, which set $\ushear=H_u(t)e_x$.  However, at no point did we impose any restrictions on $H_u$. The purpose of the second step is to show that with an application of Lemma~\ref{lem:compact:truncation}, stated next, on the set $\mathbb{T}\times[0,1]$ and a smart choice of $H_u$, we can control the support of $B$ at later times.

\begin{lemma}[Truncating to a compactly supported data]\label{lem:compact:truncation}
Given an $H^r$ divergence-free and $\mathbb{T}$-periodic vector field $B\colon\mathbb{T}\times[0,1]\rightarrow\mathbb{R}^2$ with 
$$\int_{\mathbb{T}\times[0,1]}B^1 = 0 \comma B^2|_{y=0,1}=0 \, , $$
there exists an $H^{r}$-regular divergence-free $\Btrunc\colon\mathbb{T}\times[0,1]\rightarrow\mathbb{R}^2$ 
which satisfies $\Btrunc^2|_{y=0,1}=0$, with $\Btrunc \equiv 0$ 
for $x\in[\sfrac74,5]$ or $x\in[-1,-\sfrac34]$, 
and $\Btrunc=B$ for $x\in[0,1]$. Moreover, 
the mapping $B\mapsto B_T$ is linear and
we have the inequality
$$  \| \Btrunc \|_{H^{r}(\mathbb{T}\times[0,1])} \lesssim \| B \|_{H^r(\mathbb{T}\times[0,1])} \, .  $$
\end{lemma}

\begin{proof}[Proof of Lemma~\ref{lem:compact:truncation}]
From \eqref{eq:psi:constant} and the assumption $\int_{\mathbb{T}\times[0,1]}B^1 = 0$, we have that the stream function $\psi$ for $B$,
as defined in \eqref{EQ07} with $u$ replaced by $B$, satisfies $\psi(x,0)=\psi(x,1)=0$ for all $x\in[0,1]$. 
Define 
$\Btrunc(x,t)=\nabla^\perp \left(\psi(x,y)\theta(x)\right)$ 
for a $\mathbb{T}$-periodic cutoff function $\theta$ which satisfies 
$\theta\equiv 1$ for $-\sfrac{1}{4}\leq x \leq \sfrac{5}{4}$ 
and $\theta \equiv 0$ for $-1\leq x\leq-\sfrac{3}{4}$ 
or $\sfrac74\leq x\leq 5$.  
Since $\psi(x,0)=\psi(x,1)=0$, we have that $\psi\theta$ is constant on $y=0,1$.  The rest of the assertions of Lemma~\ref{lem:compact:truncation} follow immediately.
\end{proof}

Now, we carry out the following.

\begin{enumerate}
    \item\label{item:one} \textbf{Compactly supported extensions of $u$ and $B$:\,} From Proposition~\ref{lem:local:one}, specifically \eqref{eq:mean:zero:B} and \eqref{eq:decomp}--\eqref{EQ08}, we have that the means of 
$u^1$ and $B^1$ 
at time $t=1$ vanish,
which is, by \eqref{eq:psi:constant}, equivalent to
averages of $u^1$ and $B^1$ over $\mathbb{T}\times[0,1]$ vanish.
Applying Lemma~\ref{lem:compact:truncation}, we can modify $B$ to achieve that
it stays the same in a neighborhood of $[0,1]\times[0,1]$ and it satisfies
$B=0$ for $x\in[\sfrac74,5]$ or $x\in[-1,-\sfrac34]$.


    \item\label{item:two} \textbf{Application of Proposition~\ref{lem:local:one} with a smart choice of $H_u$:\,} 
Our methodology for the proof of Proposition~\ref{lem:local:one} involved the ansatz $u=\tilde u + \ushear$, cf.~\eqref{eq:decomp}--\eqref{EQ08}. Since the mean of $u$ now vanishes, we can take $H_u$ to vanish at time $t=1$. Since the mean of $B$ now vanishes as well, we can take $\Bshear$ to vanish uniformly in time. 
For simplicity, we re-parametrize time so that the old time $t=1$ is now $t=0$.  By a sufficiently strong $\varepsilon$-rescaling of the initial data cf.~\eqref{eq:further:simp} and Lemma~\ref{lem:one}, which asserts that the perturbation $\tilde u$ around $\ushear=H_u(t)e_x$ is smaller than $\varepsilon$, we can ensure that the Lagrangian trajectories of the full velocity $u=\ushear + \tilde u$ are ``within $\varepsilon$'' of the trajectories of $\ushear$. Specifically, choose $\ushear$ so that every point in the domain $[-1,5]\times[0,1]$ moves (monotonically and periodically) to the right by distance exactly $2.5$. Lemma~\ref{lem:one} then ensures that the trajectories of $u=\ushear+\tilde u$ move to the right by at least $2.25$ and at most $2.75$. 
    \item\label{item:three} \textbf{$B$ is expelled from the set $[0,1]^2$:\,} From the choice of $\Bshear\equiv 0$, we have that \eqref{eq:B:pert:one} now reads
    \begin{equation}\label{eq:vector:transport}
    \partial_t \tilde B + u\cdot \nabla \tilde B - \tilde B\cdot \nabla u = 0  \, .  
    \end{equation}
    This vector transport equation ensures that the support of $\tilde B$ follows the Lagrangian trajectories of $u$. From the previous item, we know that the Lagrangian trajectories of $u$ move to the right by at least $2.25$ and at most $2.75$.  Therefore, the support of $\tilde B$ after the application of Proposition~\ref{lem:local:one} with this choice of $H_u$ ensures that the support of $B$ moves from the set $\{-\sfrac{3}{4}\leq x \leq \sfrac{7}{4}\}\times[0,1]$ at time $t=0$ to the set $\{1.5 \leq x \leq 4.5 \}\times[0,1]$ at time $t=1$,
from where we conclude that
$\supp B\subset\cup_{m\in\mathbb{Z}}(1.5+6m,4.5+6m)$ at $t=1$.
Thus the magnetic field $\tilde B$ now vanishes inside the set $[0,1]^2$ at time $t=1$.
\end{enumerate}

We now provide some commentary explaining why \eqref{eq:mean:alt} is not sufficient for our method of proof. Note that from the assumption that $B^1$ has vanishing average in item~1, we have that $B=\tilde B$ in item~3. But  $B^1$ would \emph{not} have vanishing average if either the initial or ending data for the control problem does not have vanishing average, and we set $\Bshear=0$ in Proposition~\ref{lem:local:one}.  In these cases, a compactly supported divergence free extension which does not penetrate the upper and lower boundaries is in general not possible. Therefore, \eqref{eq:vector:transport} would now read 
$$  \partial_t \tilde B + u\cdot\nabla\tilde B - (B_E e_x+\tilde B)\cdot\nabla u = 0 \, , $$
where $B_E$ is the average of $B^1$, which is non-zero and preserved in time. This new equation does not transport the support of $\tilde B$, and so it is not clear how to ensure that $\tilde B$ leaves the domain $[0,1]^2$. Even if one could ensure that $\tilde B$ leaves this domain so that $B=B_Ee_x$ at some later time, this property would not be preserved upon application of a control method to the remaining part of the velocity. Control methods for Euler connect the desired initial and terminal data through some common state halfway through the time interval, usually the $0$ state; a reasonable guess for the MHD analogue would be to connect both states through $(u,B)=(0,B_Ee_x)$.  Since we cannot send the initial or terminal data to this state, and there is no obvious alternative, we instead connect the initial and terminal data through $(0,0)$, thus necessitating a forcing term in the equation for $B$.

\subsection{Third step: control for Euler}\label{ss:three}
Now that the magnetic field vanishes outside of the domain $[0,1]^2$,
we truncate the extended domain $[-1,5]\times[0,1]$ back to $[0,1]^2$.  On $[0,1]^2$, we now have a vector field $u$ which does not necessarily vanish, but a magnetic field $B$ which vanishes.  Solving the MHD equations on $[0,1]^2$ with vanishing data for the magnetic field is clearly equivalent to solving the Euler equations on $[0,1]^2$. So applying any control method for Euler 
(\cite{coron1,coron2,Glass00})
will finish the proof.




\nocite{Caflisch97}
\bibliographystyle{abbrv}
\bibliography{MHD}

\end{document}